\documentclass[letterpaper, 10 pt, conference]{ieeeconf}  

\IEEEoverridecommandlockouts                              
\usepackage{cite}
\usepackage{amsmath,amssymb,amsfonts}
\usepackage{algorithmic}
\usepackage{graphicx}
\usepackage{textcomp}
\usepackage{xcolor}
\usepackage{graphics}
\usepackage{epsfig}
\usepackage{subfig}
\usepackage{float}
\usepackage{color}
\usepackage{subfig}
\usepackage{amsfonts}
\usepackage{epstopdf}
\usepackage{amssymb}
\usepackage{mathptmx}
\usepackage{ntheorem}
\usepackage{cleveref}

\newtheorem{assm}{Assumption}

\newtheorem{lema}{Lemma}
\newtheorem{thme}{Theorem}

\title{\LARGE \bf
 DeepONet of Dynamic Event-Triggered Backstepping Boundary Control for Reaction-Diffusion PDEs 
}

\author{Hongpeng Yuan, Ji Wang and Mamadou Diagne
\thanks{This work was supported in by the National Natural Science
Foundation of China under Grant 62203372.}
\thanks{Hongpeng Yuan and Ji Wang are with Department of Automation, Xiamen University, Xiamen, Fujian 361005, China; Mamadou Diagne is with Department of Mechanical and Aerospace Engineering, UC San Diego, 9500 Gilman Drive, La Jolla, CA, 92093-0411 (e-mail: \tt\small 23220221151757@stu.xmu.edu.cn; \tt\small jiwang@xmu.edu.cn; mdiagne@ucsd.edu).}
}

\begin{document}

\maketitle
\thispagestyle{empty}
\pagestyle{empty}

\begin{abstract}

We present an event-triggered  boundary control scheme for a class of reaction-diffusion PDEs using operator learning and backstepping method. Our first-of-its-kind contribution aims at learning the backstepping kernels, which inherently induces the learning of the gains in the event trigger and the control law. 
 The kernel functions in constructing the control law are approximated with neural operators (NOs) to improve the computational efficiency. Then, a dynamic event-triggering mechanism is designed, based on the plant and the continuous-in-time control law using kernels given by NOs, to determine the updating times of the actuation signal. In the resulting event-based closed-loop system, a strictly positive lower bound of the minimal dwell time is found, which is independent of initial conditions. As a result, the absence of a Zeno behavior is guaranteed. Besides, exponential convergence to zero of  the $L^2$ norm of the reaction-diffusion PDE state and the dynamic variable in the event-triggering mechanism is proved via Lyapunov analysis. The effectiveness of the proposed method is illustrated by numerical simulation.

\end{abstract}

\section{INTRODUCTION}

Event-triggered control (ETC) is a control implementation technique where the control input is updated aperiodically (only when necessary), different from the periodic sampled-data control. An ETC system comprises a stability-preserving feedback control law and a triggering mechanism that determines when updates are applied. To prevent Zeno behavior, which leads to infinitely frequent updates in a finite time frame, a minimum time separation between consecutive events must be enforced. By triggering control updates only when needed, ETC optimizes bandwidth usage, lowers computational load, and conserves energy, making it especially beneficial for networked control systems. Some notable ETC works developed for partial differential equations (PDEs) can be seen in \cite{9978722}
for state-feedback form and \cite{9477051}, \cite{RATHNAYAKE2022110026}, \cite{9319184} for output-feedback control, with applications in traffic control, re-entrant 
manufacturing systems, and Stefan problem, shown in \cite{espitia2022traffic}, \cite{diagne2021event}, and \cite{rathnayake2023observer}, respectively.

In the past few decades, an increasing number of learning-based such as reinforcement learning (RL) \cite{yu2021reinforcement}, physics-informed neural networks (PINNs) \cite{garcia2023control} and operator learning methods, including DeepONets and Fourier Neural Operators methods have been applied to solve PDEs and their related control problems, offering advantages such as improved computational efficiency and adaptability to high-dimensional systems.  These methods have been exploited to develop optimal, adaptive, and robust control algorithms for complex nonlinear systems or infinite-dimensional systems, where traditional methods face challenging computational limitations. The present work pertains to DeepONet, which has been proven to speed up the computation of PDE backstepping gain kernel functions. The method has been applied to both reaction-diffusion \cite{krstic2024neural, wang2025deep} and  hyperbolic \cite{wang2023backstepping,zhang2024mitigating,lamarque2024adaptive} PDEs. More precisely, the function-to-function mapping, which is encapsulated into NOs for the computation of the backstepping gain kernel PDEs derived model-based stabilizing control law,  enables the re-computation of the controller gain functions via Neural Networks (NN) when variations occur in the plant functional parameters. DeepONet theory expands the classical "universal approximation theorem" for functions \cite{hornik1989multilayer} by demonstrating that it also holds for nonlinear operators, thereby establishing a universal approximation framework for operator learning \cite{chen1995universal,lu2021learning}.

  In this paper, we develop the computation of backstepping gain kernel to be exploited for the construction of a NO-approximated event-triggered boundary control design for the reaction-diffusion PDE with spatially varying coefficient. The operator to be approximated is given in the form of a hyperbolic-Goursat PDE and the spatially-varying reactivity represent the input of the function-to-function mapping.  The neural operator, used to approximate the kernel, is employed  in the design of the triggering mechanism and lead to an approximated triggering sequence. More precisely, the error of NO-based continuous-in-time controller and piecewise-constant controller is used in the design of event trigger mechanism.  The use of NO  in ETC design significantly enhances the computational efficiency, meanwhile no Zeno behavior and exponential regulation are guaranteed. As we prove $L^2$ exponential convergence with the approximated gains, the DeepONet ETC framework retains the stability property obtained with exact gain kernels. However, the speed of convergence of the closed-loop system is subject to the training data size: more data leads to faster convergence.   To the
  best of the authors’ knowledge, this is the first  
    study about event-triggered control of
  the reaction-diffusion PDEs using neural operators. 

The paper is organized as follows. Section \ref{sec2} presents the   problem formulation and briefly recall the continuous-time controller.  In Section \ref{sec3},  introduce the NO-approximated continuous-in-time control law while Section \ref{sec4} presents the event-triggered NO-approximated control law. The main result is stated in Section \ref{sec5} and illustrative simulation results are presented  in Section \ref{sec6}. The paper ends with concluding remarks  in Section \ref{sec7}.

\textbf{Notation:} The symbol $\mathbb{N}$ denotes the set of natural numbers including zero, and the notation $\mathbb{N}^*$ for the set of natural numbers without 0. We also denote $\mathbb{R}_{+}:=[0,+\infty)$ and $\mathbb{R}_{-}:=(-\infty, 0)$. We use the notation $f[t]$ (e.g., $u,w$) to denote the profile of $f$ at certain $t \geq 0$, i.e., $(f[t])(x)=$ $f(x, t)$ .

\section{Problem Formulation and Continuous-in-time Control Design }\label{sec2} 

\subsection{Problem statement }
Consider the following plant
\begin{align}
	& u_t(x, t)=\varepsilon u_{x x}(x, t)+\lambda(x) u(x, t), \label{1}\\
	&  u_x(0, t)=0,\label{2} \\
	& u_x(1, t)+qu(1,t)=U(t), \label{xu}
\end{align}
$\forall ( x,t) \in [0,1]\times[0,\infty)$, where $U(t)$ is the control input to be designed, and where $u(x,t)$ is the state of the reaction-diffusion PDE.  
The parameters $\varepsilon,   q$ in the reaction-diffusion PDE are positive, and $\lambda \in C^2\left([0,1] ; \mathbb{R}_{+}\right)$. 
\begin{assm}
	The parameters $q, \varepsilon>0, \lambda \in C^2\left([0,1] ; \mathbb{R}_{+}\right)$ satisfy the following relation:
	\begin{align}
		q>\frac{\lambda_{\max }}{2 \varepsilon}+\frac{1}{2}, \label{qlt}
	\end{align}
	where
	\begin{align}
		\lambda_{\max } \triangleq \max _{x \in[0,1]} \lambda(x) .
	\end{align}
\end{assm} 
Assumption 1 is important in ensuring the stability of the
target system under PDE backstepping control with dynamic
event-triggering. According to Assumption 1 we can avoid using
the signal $ u(1, t)$ in the nominal control law. Such avoidance is
crucial for dynamic ETC design due to the challenges associated
with obtaining a meaningful bound on the rate of change of $u(1, t)$. An eigenfunction expansion of the solution of (1)-(3) with $U(t)=0$ shows that the system is unstable when $\min _{x \in[0,1]} \lambda(x)>\varepsilon \pi^2 / 4$.
 
\subsection{Continuous-in-time Control Design }

 In this part, we design a continuous-in-time boundary control law $U(t)$ in \eqref{xu}. Consider the invertible backstepping transformation
\begin{align}
	{w}(x, t)= &  {u}(x, t)-\int_0^x k(x, y)  {u}(y, t) \mathrm{d} y, \label{10} 
\end{align}
where $k(x,y)$ is given by
\begin{align}
	k_{x x}(x, y)-k_{y y}(x, y) & =\frac{\lambda(y)}{\varepsilon} k(x, y), \label{k7}\\
	  k_y(x, 0) & = 0, \\
	k(x, x) & =-\frac{1}{2 \varepsilon} \int_0^x \lambda(y) d y. \label{k9}
\end{align}
  The kernel equations \eqref{k7}--\eqref{k9} admit a unique solution on
  the triangular domain $\mathcal{T}={0 \leq y \leq x \leq 1}$ \cite{rathnayake2025performance}. Applying the backstepping transformation \eqref{10} into the original system \eqref{1}--\eqref{xu}  with choosing the control law in \eqref{xu} as $U_f(t)$ defined as
\begin{align}
	U_f(t)=\int_{0}^{1} K(y) {u}(y,t)dy\label{Ufn}
\end{align}
where
\begin{align}
	K(y)=\wp k(1, y)+k_x(1, y)
\end{align}
with
\begin{align}
\wp=q-\frac{1}{2 \varepsilon} \int_0^1 \lambda(y) d y,
\end{align}
we arrive at the following target system:
\begin{align}
	w_t(x, t) & =\varepsilon w_{x x}(x, t), \\
  w_x(0, t) & = 0, \\
	w_x(1, t) & =-\wp w(1, t).
\end{align}
Recalling \eqref{qlt}, we can get 
\begin{align}
	\wp>\frac{1}{2}. \label{wp}
\end{align}
The inverse transformation of \eqref{10} is given by
\begin{align}
	u(x, t)=w(x, t)+\int_0^x l(x, y) w(y, t) d y
\end{align}
where $l(x, y)$ satisfies
\begin{align}
	l_{x x}(x, y)-l_{y y}(x, y) & =-\frac{\lambda(x)}{\varepsilon} l(x, y), \\
	 l_y(x, 0) & = 0 ,\\
	l(x, x) & =-\frac{1}{2 \varepsilon} \int_0^x \lambda(y) d y.
\end{align}

\section{NO-approximated Continuous-in-time Control Design}\label{sec3}

Neural operators can be used to approximate the operator mapping of functions. In this section, we introduce the neural operator using DeepONet to approximate the mapping from the nominal system parameters to the backstepping kernels. 
 The kernel operator $\mathcal{K}: C^2 ([0,1] ; \mathbb{R}_{+} )   \rightarrow  {C}^2(\mathcal{T}) $ is defined by
\begin{align}
	\mathcal{K}\left(\lambda\right)(x, y)=:\left( k(x, y) \right) \label{calk}
\end{align}
Consider a nonlinear mapping $\mathcal{G}: \mathcal{U} \mapsto \mathcal{V}$, where $\mathcal{U}$ and $\mathcal{V}$ are function spaces. Its neural operator approximation can be defined as
\begin{align}
	\mathcal{G}_{\mathbb{N}}\left(\mathbf{u}_n\right)(y)=\sum_{k=1}^p  {g^{\mathcal{N}}\left(\mathbf{u}_n ; \vartheta^{(k)}\right)}  {f^{\mathcal{N}}\left(y ; \theta^{(k)}\right)},
\end{align}
where $\mathbf{u}_n$ is the evaluation of function $u \in \mathcal{U}$ at points $x_i=x_1, \ldots, x_n,$ $\ p$ is the number of basis components in the target space, $y \in Y$ is the location of the output function $v(y)$ evaluations, and $g^{\mathcal{N}}, f^{\mathcal{N}}$ are NNs termed branch and trunk networks.

\begin{thme} \label{theorem1}
	(DeepONet universal approximation theorem)  Let $X \subset \mathbb{R}^{d_x}, Y \subset \mathbb{R}^{d_y}$ be compact sets of vectors $x \in X$ and $y \in Y, d_x, d_y \in \mathbb{N}$. Let $\mathcal{U}: X \rightarrow$ $\mathbb{U} \subset \mathbb{R}^{d_u}$ and $\mathcal{V}: Y \rightarrow \mathbb{V} \subset \mathbb{R}^{d_v}$ be sets of continuous functions $\mathrm{u}(x)$ and $\mathrm{v}(y), d_{\mathrm{u}}, d_{\mathrm{v}} \in \mathbb{N}$. Assume the operator $\mathcal{G}: \mathcal{U} \rightarrow \mathcal{V}$ is continuous. Then, for all $\iota>0$, there exists a $m^{\star}, p^{\star} \in \mathbb{N}$ such that for each $m \geq m^{\star}, p \geq p^{\star}$, there exist $\theta^{(k)}, \vartheta^{(k)}$, neural networks $f^{\mathcal{N}}\left(\cdot ; \theta^{(k)}\right), g^ {\mathcal{N}}\left(\cdot ; \vartheta^{(k)}\right), k=$ $1, \ldots, p$ and $x_j \in X, j=1, \ldots, m$, with corresponding $\mathbf{u}_m=\left( {u}\left(x_1\right),  {u}\left(x_2\right), \cdots,  {u}\left(x_m\right)\right)^{\top}$, such that
	
	$$
 \left|\mathcal{G}(\mathbf{u})(y)-\mathcal{G}_{\mathbb{N}}\left(\mathbf{u}_m\right)(y)\right|<\iota,
	$$	
	for all functions $\mathbf{u} \in \mathcal{U}$ and all values $y \in Y$ of $\mathcal{G}(\mathbf{u})(y) \in \mathcal{V}$.
\end{thme}
	The kernel operator $\mathcal{K}$ maps the system parameters to the backstepping kernels, such that there exists a neural operator approximates the kernel operator $\mathcal{K}$, then we have the following lemma:
 \begin{lema} \label{lema1}
 	For all $\iota>0$, there exists a neural operator $\hat{\mathcal{K}}$ such that  for all $(x, y) \in \mathcal{T}$
 	\begin{align}
 		& |\mathcal{K}(\lambda)(x, y)-\hat{\mathcal{K}}(\lambda)(x, y)| \notag\\
 		& +\left|2 \frac{d}{d x}(\mathcal{K}(\lambda)(x, x)-\hat{\mathcal{K}}(\lambda)(x, x))\right| \notag\\
 		& +|\left(\partial_{x x}-\partial_{y y}\right)(\mathcal{K}(\lambda)(x, y)-\hat{\mathcal{K}}(\lambda)(x, y)) \notag\\
 		& -\lambda(y)(\mathcal{K}(\lambda)(x, y)-\hat{\mathcal{K}}(\lambda)(x, y)) |<\iota .\label{supk}
 	\end{align}

 \end{lema}

Proof: The system \eqref{k7}-\eqref{k9} has a unique $ {C}^2(\mathcal{T})$ solution, therefore the neural operator $\hat{\mathcal{K}}\left(\lambda\right)(x, y)$ could approximate the kernels for given nominal parameters. Using Theorem \ref{theorem1}, we can obtain a maximum approximation error defined as $\iota$. This finishes the proof of Lemma \ref{lema1}.
\begin{thme}
	 There exists a sufficiently small $\iota^*  >0$ such that the feedback law
	\begin{align}
			U_{NO}(t)=\int_0^1 \hat{K}(y) u(y, t) d y \label{NOcontrol}
	\end{align}	
	where $\hat{K}(y)=\wp \hat{k}(1, y)+\hat{k}_x(1, y)$ with NO gain kernels $\hat{k}=\hat{\mathcal{K}}(\lambda)$ of approximation accuracy $\iota \in\left(0, \iota^*\right)$ in relation to the exact backstepping kernel $k=\mathcal{K}(\lambda)    $ ensures that the closed-loop system satisfies the exponential stability bound
\begin{align}
		\|u[t]\| \leq M \mathrm{e}^{-\frac{\sqrt{\varepsilon}}{2}t }\left\|u[0]\right\|,  \label{um0}
\end{align}	
	where
	\begin{align}
			M(\iota, \bar{\lambda})=\left(1+\iota+\frac{2\bar{\lambda}}{\varepsilon} \mathrm{e}^{ \frac{4\bar{\lambda}}{\varepsilon}}\right)\bigg(1+(\frac{2\bar{\lambda}}{\varepsilon} \mathrm{e}^{ \frac{4\bar{\lambda}}{\varepsilon}}+\iota) \mathrm{e}^{\frac{2\bar{\lambda}}{\varepsilon} \mathrm{e}^{ \frac{4\bar{\lambda}}{\varepsilon}}+\iota}\bigg),
	\end{align}
    where
\begin{align}
	\bar{\lambda} \triangleq \max _{x \in[0,1]} |\lambda(x)| . \label{blambda}
\end{align}
\end{thme}
\begin{proof}
 Take the backstepping transformation
\begin{align}
\hat{w}(x, t)=u(x, t)-\int_0^x \hat{k}(x, y) u(y, t) d y.  \label{bretrans}
\end{align}
 With the control law \eqref{NOcontrol}, the target system becomes
\begin{align}
	\hat{w}_t(x, t)= &\varepsilon \hat{w}_{x x}(x, t)+\delta_{k 0}(x) u(x, t) \notag \\
	& +\int_0^x \delta_{k 1}(x, y) u(y, t) d y\\
	 \hat{w}_x(0, t)=& (k(0,0)-\hat{k}(0,0)) \hat{w}(0,t), \\
	 \hat{w}_x(1, t)=&-\wp \hat{w}(1, t)+(k(1,1)-\hat{k}(1,1)) u(1,t),
\end{align}
with
\begin{align}\
	\delta_{k 0}(x) & =2\varepsilon \frac{d}{d x}(\hat{k}(x, x))+\lambda(x)\notag \\
	& =-2\varepsilon \frac{d}{d x}(\tilde{k}(x, x)), \\
	\delta_{k 1}(x, y) & =\varepsilon \partial_{x x} \hat{k}(x, y)-\varepsilon \partial_{y y} \hat{k}(x, y)-\lambda(y) \hat{k}(x, y) \notag	\\
	& =-\varepsilon \partial_{x x} \tilde{k}(x, y)+\varepsilon \partial_{y y} \tilde{k}(x, y)+\lambda(y) \tilde{k}(x, y),
\end{align}
where
\begin{align}
	\tilde{k}(x,y)=k(x,y)-\hat{k}(x,y).
\end{align}
Using \eqref{supk}, recalling \eqref{calk}, we get
\begin{align}
	\|\delta_{k 0}\|_{\infty} \leq \iota, \label{delinf1}\\
	\|\delta_{k 1}\|_{\infty}  \leq \iota. \label{delinf2}
\end{align}
The inverse transformation $  \hat{w} \mapsto  {u}$ is given in the form
\begin{align}
	{u}(x, t)=\hat{w}(x, t)+\int_0^x \hat{l}(x, y) \hat{w}(y, t) d y . \label{xit}
\end{align}
We derive from \eqref{xit} that
\begin{align}
	u^2(1,t) \leq 2 \hat{w}^2(1,t) + q_0 \|\hat{w}[t]\|^2 ,\label{37}
\end{align}
where
\begin{align}
	q_0= & 2 \int_{0}^{1} (\hat{l}(1,y))^2 dy . \label{q0}
\end{align}
It is shown in \cite{doi:10.1137/1.9780898718607} that the direct and inverse backstepping kernels satisfy in general the relationship
\begin{align}
	\hat{l}(x, y)=\hat{k}(x, y)+\int_y^x \hat{k}(x, \xi) \hat{l}(\xi, y) d y .
\end{align}
The inverse kernel satisfies the following conservative bound
\begin{align}
	\|\hat{l}\|_{\infty} \leq\|\hat{k}\|_{\infty} \|  e^{\|\hat{k}\|_{\infty}}
\end{align}
Since $\|k-\hat{k}\|_{\infty}<\iota$, we have that $\|\hat{k}\|_{\infty} \leq\|k\|_{\infty}+\iota$. According to \cite{1369395},  we can get
\begin{align}
	\lvert k(x,y)\rvert& \leq \frac{2\bar{\lambda}}{\varepsilon} \mathrm{e}^{ \frac{4\bar{\lambda}}{\varepsilon}}  \label{42q}
\end{align}
and hence
\begin{align}
	\|\hat{l}\|_{\infty} \leq\left(\frac{2\bar{\lambda}}{\varepsilon} \mathrm{e}^{ \frac{4\bar{\lambda}}{\varepsilon}}+\iota\right) \mathrm{e}^{\frac{2\bar{\lambda}}{\varepsilon} \mathrm{e}^{ \frac{4\bar{\lambda}}{\varepsilon}}+\iota}.  \label{hatlinf1}
\end{align}
The Lyapunov function
\begin{align}
	V(t)=\frac{1}{2} \int_0^1 \hat{w}^2(x, t) d x 
\end{align}
has a derivative
\begin{align}
\dot{V}(t)=&-\varepsilon \wp \hat{w}^2(1,t)+\varepsilon \tilde{k}(1,1) \hat{w}(1,t) u(1,t)  \notag\\
&-\varepsilon \tilde{k}(0,0) \hat{w}^2(0,t) - \varepsilon \|\hat{w}_x\|^2+ \Delta_0(t)+ \Delta_1(t) \label{v1t}
\end{align}
where
\begin{align}
	& \Delta_0(t)=\int_0^1 \hat{w}(x, t) \delta_{k 0}(x) u(x, t) d x, \\
	& \Delta_1(t)=\int_0^1 \hat{w}(x, t) \int_0^x \delta_{k 1}(x, y) u(y, t) d y d x.
\end{align}
Using \eqref{delinf1}, \eqref{delinf2},  and \eqref{hatlinf1} we get
\begin{align}
	\Delta_0 & \leq\left\|\delta_{k 0}\right\|_{\infty}\left(1+\|\hat{l}\|_{\infty}\right)\|\hat{w}\|^2  \label{delta0}
\end{align}
and
\begin{align}
	\Delta_1= & \int_0^1 \hat{w}(x) \int_0^y \hat{w}(y) \int_y^x \delta_{k1}(x, \sigma) \hat{l}(\sigma, y) d \sigma d y d x\notag \\
	& +\int_0^1 \hat{w}(x) \int_0^x \delta_{k1}(x, y) \hat{w}(y) d y d x\notag\\
	\leq	& \left\|\delta_{k 1}\right\|_{\infty}\left(1+\|\hat{l}\|_{\infty}\right)\|\hat{w}\|^2
\end{align}
From Agmon’s and Young’s inequalities, we have that
\begin{align}
	\hat{w}^2(0,t) \leq \hat{w}^2(1,t) +\|\hat{w}[t]\|^2 +\|\hat{w}_x[t]\|^2. \label{51}
\end{align}
From Poincaré inequality, we have that
\begin{align}
	 -\left\|\hat{w}_x[t]\right\|^2 \leq \frac{1}{2} \hat{w}^2(1, t)-\frac{1}{4}\|\hat{w}[t]\|^2 \label{poinc}
\end{align}
Using \eqref{supk}, \eqref{37}, \eqref{hatlinf1}, \eqref{v1t}--\eqref{poinc}, and Young's  inequality, we get
\begin{align}
	\dot{V} \leq& -( \varepsilon \wp-\frac{5 \varepsilon \iota }{2}-\frac{\varepsilon}{2}) \hat{w}^2(1,t)-(  \frac{\varepsilon}{4}-  \delta^* )\|\hat{w}[t]\|^2 
\end{align}
where
\begin{align}
	\delta^*(\iota, \bar{\lambda})=2 \iota \bigg(1+\left(\frac{2\bar{\lambda}}{\varepsilon} \mathrm{e}^{ \frac{4\bar{\lambda}}{\varepsilon}}+\iota\right) \mathrm{e}^{\frac{2\bar{\lambda}}{\varepsilon} \mathrm{e}^{ \frac{4\bar{\lambda}}{\varepsilon}}+\iota} \bigg)+ \varepsilon \iota (\frac{q_0}{2}+\frac{5}{4}) \label{deltastar}
\end{align}
is an increasing function of $\iota, \bar{\lambda}$, with the property that $\delta^*(0, \bar{\lambda})=$ 0 . Recalling \eqref{wp}, there exists $\iota^*(\bar{\lambda})$ such that, for all $\iota \in\left[0, \iota^*\right]$,
\begin{align}
	\dot{V} \leq- \frac{\varepsilon}{4} V,
\end{align}
namely, $V(t)\leq V_0 e^{-\frac{\varepsilon}{4} t}$. From the direct and inverse backstepping transformations it follows that
\begin{align}
	\frac{1}{1+\|\hat{l}\|_{\infty}}\|u\| \leq \|\hat{w}\| \leq\left(1+\|\hat{k}\|_{\infty}\right)\|u\|. \label{ukw}
\end{align}
Thus we can get
$
	\|u[t]\| \leq (1+\|\hat{l}\|_{\infty}) ((1+\|\hat{k}\|_{\infty}) e^{-\frac{\sqrt{\varepsilon}}{2}t} \|u[0]\| ,
$ and recalling \eqref{42q} and \eqref{hatlinf1}, we derive \eqref{um0}.
\end{proof}
\section{ Event-Triggered NO-approximated Boundary Control Design}\label{sec4}
We strive to stabilize the closed-loop system \eqref{1}--\eqref{xu} while sampling the continuous-in-time controller $U_{NO}(t)$ given by \eqref{Ufn} at a certain sequence of time
instants $\left(t_j\right)_{j \in \mathbb{N}}$. These time instants will be given a precise
characterization later based on an event trigger. The control
input is held constant between two successive time instants and
is updated when a certain condition is met. Therefore, we define
the control input for $t \in  [t_j , t_{j+1}), j \in  \mathbb{N}$, as
\begin{align}
	U_{d} \ :=		&U_{NO}(t_j) \notag\\
	=&\int_{0}^{1} \hat{K}(y) {u}(y,t_j)dy.\label{c16}
\end{align}
Inserting the piecewise-constant control input $U_{d }$ into \eqref{xu}, the boundary condition becomes
$
u_x(1, t)+qu(1,t)= U_d .$
Define the difference between the continuous-in-time control signal $U_{NO}(t)$ in \eqref{Ufn} and the event-triggered control input $U_d $
in \eqref{c16} as $d(t)$, given by
\begin{align}
	d(t)\ :=&U_{NO}(t)-U_{d}\notag\\
	=&\int_{0}^{1} \hat{K}(y) ({u}(y,t)-{u}(y,t_j))dy,\label{c17}
\end{align}
for $t \in\left[t_j, t_{j+1}\right)$, which will be used in building the ETM.\\
The sequence of time instants $I=\left\{t_0, t_1, t_2, \ldots\right\}\ (t_0=0)$ is defined as (for $j\in \mathbb{N}$):\\
a) if $\left\{t \in \mathbb{R}_{+} \mid t>t_j \wedge d^2(t)>-\xi m(t)\right\}=\emptyset$, then the set of the times of the events is $\left\{t_0, \ldots, t_j\right\}$,\\
b) if $\left\{t \in \mathbb{R}_{+} \mid t>t_j \wedge d^2(t)>-\xi m(t)\right\} \neq \emptyset$, then the next event time is given by
\begin{equation}
	t_{j+1}=\inf \left\{t>t_j: d(t)^2 \geq-\xi m(t)\right\},\label{c20}
\end{equation}
where the positive constant $\xi$ is a design parameter and the dynamic variable $m(t)$ in \eqref{c20} satisfies the ordinary differential equation,
\begin{align}
	\dot{m}(t)= & -\eta m(t)+\lambda_d d(t)^2-\kappa_1 \| {u}[t]\|^2\notag\\
	& -\kappa_2 u^2(1,t)-\kappa_3 u^2(0,t) \label{c21}
\end{align}
for $t \in\left(t_j, t_{j+1}\right)$ with $m\left(t_0\right)=m(0)<0$.   {\color{black}The design parameter $\eta>0$ is free} and the positive design parameters $\kappa_1, \kappa_2, \kappa_3, \lambda_d  $ are to be determined later. It is worth noting that the initial condition for $m(t)$ in each time interval has been chosen such that $m(t)$ is time-continuous. Therefore, we define $m(t_j^{-})=m(t_j)=m(t_j^{+})$. Recalling  \eqref{c20} and considering the time-continuity of $m(t)$, we can
obtain the following estimate:
\begin{align}
	m(t)\leq & m(t_j) e^{-\left(\eta+\lambda_d \xi\right) (t-t_j)}  -\int_{t_j}^t e^{-\left(\eta+\lambda_d \xi\right)(t-\varepsilon)}\left(\kappa_1 \|u[\varepsilon]\|^2\right. \notag\\
	& \left.+\kappa_2  u^2(1,\varepsilon) +\kappa_3  u^2(0,\varepsilon)\right) d \varepsilon \label{c24}
\end{align}
for $t\in [t_j,t_{j+1}],\ j\in \mathbb{N}$. We choose $m(0)<0$ such that  $m(t)<0$ for $t\in [0,t_1]$ and by recursion we derive $m(t)<0$ all the time. For some design parameters $\kappa_1, \kappa_2,\kappa_3 $, the minimal dwell
time under the proposed ETM \eqref{c20} is larger than a positive constant, i.e., no Zeno behavior, which will be proved in Lemma \ref{lema3}.
\section{Main Result}\label{sec5}
The block diagram of the closed-loop system consisting of the plant, the neural operator, the controller,	and the event trigger is presented
in Fig. 1. The main result of the event trigger design are shown as follows.
\begin{figure}
	\centering
	\includegraphics[width=9cm]{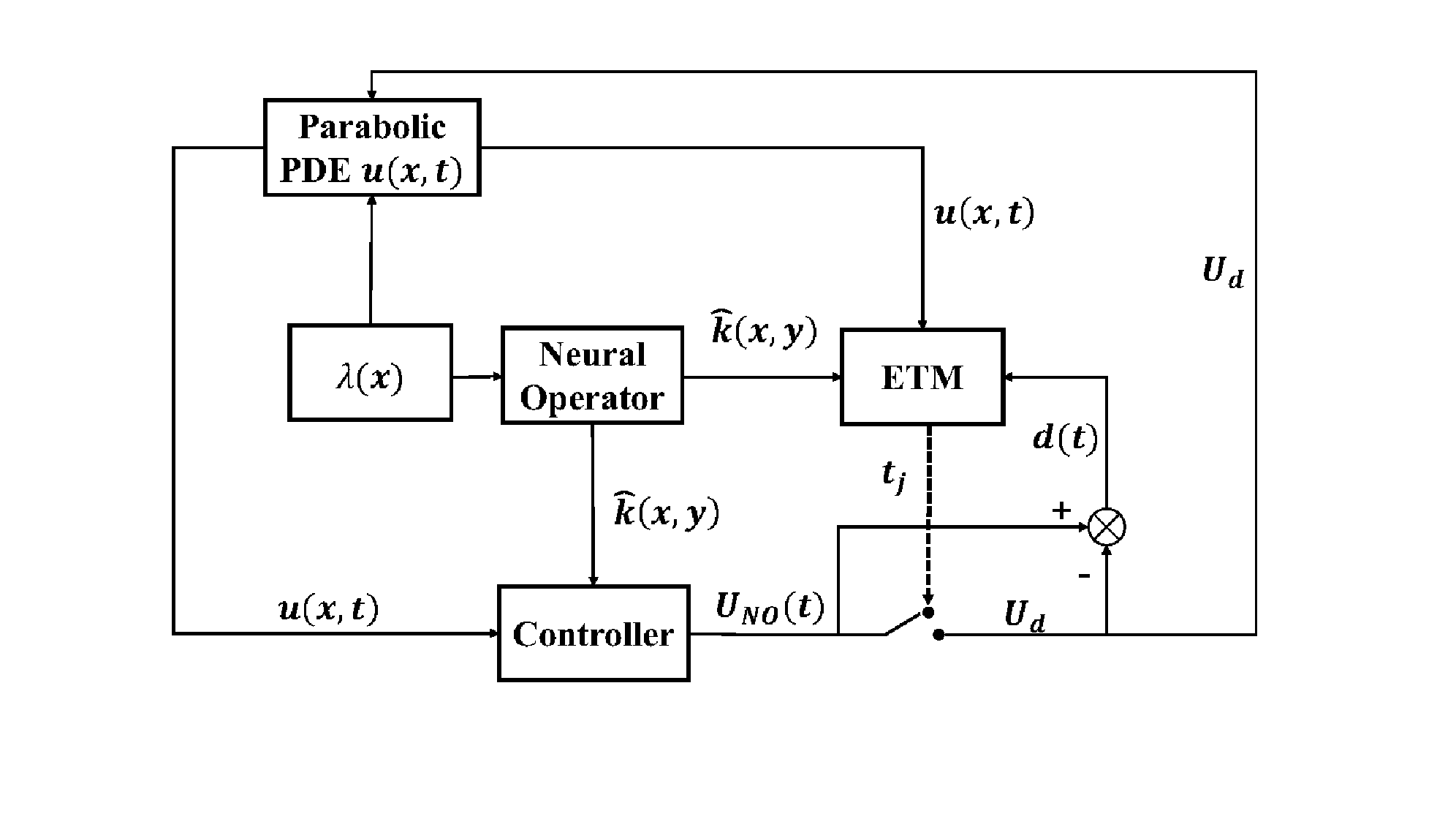}\vspace{-0.5cm}
	\caption{Block diagram of the closed-loop system.}
	\label{fig:block}
\end{figure}
\newtheorem{theorem}{Theorem}
\begin{thme} \label{theo1}
	For all initial conditions $u[0] \in L^2(0,1)  $ and $m(0) \in \mathbb{R}_{-}$, there exists a sufficiently small $\iota^*>0$ such that for approximation accuracy $\iota \in (0, \iota^*)$,	the closed-loop system, which consists of the plant \eqref{1}--\eqref{xu} and
	the event-triggered control law \eqref{c16}  with the event-triggering
	mechanism \eqref{c20}, \eqref{c21}, has the following properties:\\
	1) No Zeno behavior.\\
	2) States are exponentially convergent to zero, i.e., 
\begin{align}
		\Omega(t) \leq \Upsilon \Omega(0) e^{-\sigma t}, \label{ome}
	\end{align}
    where
    \begin{align}
    &\Omega(t) =  \|u[t]\|^2 +|m(t)|, \label{Omega}\\
    &\Upsilon= \left(1+(\iota+\frac{2\bar{\lambda}}{\varepsilon} \mathrm{e}^{ \frac{4\bar{\lambda}}{\varepsilon}})^2\right)\bigg(1+(\frac{2\bar{\lambda}}{\varepsilon} \mathrm{e}^{ \frac{4\bar{\lambda}}{\varepsilon}}+\iota)^2 \mathrm{e}^{\frac{4\bar{\lambda}}{\varepsilon} \mathrm{e}^{ \frac{4\bar{\lambda}}{\varepsilon}}+2\iota}\bigg),
    \end{align}
    and where
    \begin{align}
       \sigma=  \min\{\eta, \sigma^* \} \label{sigma1}
    \end{align}
     with
    \begin{align}
        \sigma^*(\iota,\bar{\lambda}) = &\frac{\varepsilon}{8}- 4 \iota \bigg(1+\left(\frac{2\bar{\lambda}}{\varepsilon} \mathrm{e}^{ \frac{4\bar{\lambda}}{\varepsilon}}+\iota\right) \mathrm{e}^{\frac{2\bar{\lambda}}{\varepsilon} \mathrm{e}^{ \frac{4\bar{\lambda}}{\varepsilon}}+\iota} \bigg)\notag\\
        &- \varepsilon \iota \bigg(2(\frac{2\bar{\lambda}}{\varepsilon} \mathrm{e}^{ \frac{4\bar{\lambda}}{\varepsilon}}+\iota)^2 \mathrm{e}^{\frac{4\bar{\lambda}}{\varepsilon} \mathrm{e}^{ \frac{4\bar{\lambda}}{\varepsilon}}+2\iota}+\frac{5}{2}\bigg),
    \end{align} 
    where $\bar{\lambda}$ is defined in \eqref{blambda}.
\end{thme}
\begin{proof}
	The detailed process is shown next.
\end{proof}
\subsection{Proof of property $1$ of Theorem $\ref{theo1}$} \label{5.1}
Taking the time derivative of \eqref{c17}, recalling \eqref{1}--\eqref{xu},  using \eqref{NOcontrol}, \eqref{c16} and \eqref{c17}, we can obtain that 	 \\
	\begin{align}
		\dot{d}(t)= & -\varepsilon \hat{K}(1) d(t)-\left(\varepsilon q \hat{K}(1)+\varepsilon \frac{d \hat{K}(x)}{d x}|_{x=1}\right)  {u}(1, t)\notag \\
		& +\int_0^1\left(\varepsilon \frac{d^2 \hat{K}(y)}{d y^2}+\varepsilon \hat{K}(1) \hat{K}(y)+\lambda(y) \hat{K}(y)\right)  {u}(y, t) d y\notag \\
		& +\varepsilon \left( \frac{d \hat{K}(x)}{d x}|_{x=0}  \right) {u}(0, t),
	\end{align}
	for $t \in \left(t_j, t_{j+1}\right)$. Therefore, we can obtain that
	\begin{align}
		\dot{d}(t)^2 \leq &\epsilon_1 d(t)^2+\epsilon_2 \|u[t]\|^2 +\epsilon_3 u^2 (1,t)+\epsilon_4 u^2 (0,t) \label{dt}
	\end{align}
	for $t \in\left(t_j, t_{j+1}\right)$, where
	\begin{align}
		\epsilon_1 =& 4\varepsilon^2 \hat{K}^2(1), \\
		\epsilon_2 = & 4\int_0^1\left(\varepsilon \frac{d^2 \hat{K}(y)}{d y^2}+\varepsilon \hat{K}(1) \hat{K}(y)+\lambda(y) \hat{K}(y)\right)  ^2 d y,\label{ep2} \\
		\epsilon_3 = & 4\left(\varepsilon q \hat{K}(1)+\varepsilon \frac{d \hat{K}(x)}{d x}|_{x=1}\right) ^2, \\
		\epsilon_4 = & 4 \varepsilon^2 \left( \frac{d \hat{K}(x)}{d x}|_{x=0}  \right) ^2 \label{ep4}
	\end{align}

\begin{lema}  \label{lema3}
	Under the event-triggered boundary controller defined in \eqref{c16}, for some positive $\kappa_1, \kappa_2,\kappa_3 $ to be chosen in \eqref{c21}, there exists a minimal dwell time ${\tau}>0$ such that $t_{j+1}-t_j > {\tau}$ for all $j \in \mathbb{N}$, which is independent of the initial conditions.
\end{lema}
\begin{proof}
Like \cite{10885941}, we choose
	\begin{align}
		\kappa_1 \geq \frac{2 \epsilon_2(\hat{k})}{\xi}, \kappa_2 \geq \frac{2 \epsilon_3(\hat{k})}{\xi},\kappa_3 \geq \frac{2 \epsilon_4(\hat{k})}{\xi} , \label{kp1}
	\end{align}
where $\epsilon_2, \ \epsilon_3, \ \epsilon_4$	are given in \eqref{ep2}--\eqref{ep4} depending on the NO gain kernels. The we can derive a lower bound for the minimal dwell time $\tau$ that
		\begin{align}
		\tau=\int_0^1 \frac{1}{\bar{n}_3+\bar{n}_2 s+\bar{n}_1 s^2} d s>0 ,\label{90}
	\end{align}
where $\bar{n}_1=\frac{1}{2} \lambda_d \xi,\ \bar{n}_2=1+\epsilon_1+\xi \lambda_d+\eta,\ \bar{n}_3=1+\eta+\epsilon_1+\frac{1}{2} \xi \lambda_d$ are positive constants.
\end{proof}
The property $1$ of Theorem \ref{theo1} is obtained.
\subsection{Proof of property $2$   of Theorem $\ref{theo1}$} \label{5.2}
  It can be shown that applying the backstepping transformation \eqref{bretrans} into the system \eqref{1}--\eqref{xu} with choosing the control law in \eqref{xu} as $U_d$  between $t_j$ and $t_{j+1}$, yields the following target system, valid for $t \in\left[t_j, t_{j+1}\right), j \in \mathbb{N}$ :
\begin{align}
	\hat{w}_t(x, t)= &\varepsilon 	\hat{w}_{x x}(x, t)+\delta_{k 0}(x) u(x, t) \notag \\
& +\int_0^x \delta_{k 1}(x, y) u(y, t) d y \label{wt1}\\
 	\hat{w}_x(0, t)=&\tilde{k}(0,0) \hat{w}(0,t), \\
	\hat{w}_x(1, t)=&-\wp 	\hat{w}(1, t)+\tilde{k}(1,1) u(1,t)-d(t). \label{wt3}
\end{align}
	Using Young’s and Cauchy–Schwarz inequalities on \eqref{xit}, we can get
	\begin{align}
		\| {u}[t]\|^2 \leq& q_1 \| \hat{w}[t] \|^2  ,\label{u3}\\
		u^2(0,t) \leq&  \hat{w}^2(1,t)+\|\hat{w}_x[t]\|^2+\|\hat{w}[t]\|^2,\label{v3}  
	\end{align}
where $q_1= (1+(\int_0^1 \int_0^x \hat{l}^2(x,y)dy dx)^{\frac{1}{2}})^2$.

Define a Lyapunov function as
\begin{align}
	V(t)=&\frac{r_0}{2} \int_0^1 \hat{w}^2(x, t) d x - m(t) \label{lyapunov}
\end{align}
where $m(t)$ is defined in \eqref{c21}. Using \eqref{c21} and \eqref{wt1}--\eqref{wt3}, we obtain
\begin{align}
	\dot{V}(t)=&-r_0 \varepsilon \wp \hat{w}^2(1,t) -r_0 \varepsilon \hat{w}(1,t) d(t) \notag\\
	&+r_0 \varepsilon \tilde{k}(1,1) \hat{w}(1,t) u(1,t)-r_0 \varepsilon \tilde{k}(0,0) \hat{w}^2(0,t) \notag\\
	&-r_0 \varepsilon \|\hat{w}_x\|^2 +r_0 \Delta_0(t)+r_0 \Delta_1(t)+\eta m(t)-\lambda_d d(t)^2\notag\\
	&+\kappa_1 \| {u}[t]\|^2 +\kappa_2 u^2(1,t)+\kappa_3 u^2(0,t)
\end{align}
for $t \in\left(t_j, t_{j+1}\right)$ .
Using \eqref{supk}, \eqref{37}, \eqref{hatlinf1}, \eqref{u3}, \eqref{v3}, \eqref{delta0}--\eqref{poinc}, and  Young's  inequality, we get
\begin{align}
	\dot{V} \leq& -(r_0 \varepsilon \wp-\kappa_2  -2 \kappa_3 -r_0 \varepsilon(\frac{1}{4}+ \delta_0+ \frac{5\iota}{2})) \hat{w}^2(1,t)\notag\\
	&-(\frac{r_0 }{8}\varepsilon-\kappa_1 q_1-\kappa_2  -\kappa_3 q_0 -r_0 \delta^* )\|\hat{w}\|^2  \notag\\
	& +\eta m(t)-(\frac{r_0}{2} \varepsilon-\kappa_2 ) \|\hat{w}_x\|^2-(\lambda_d-\frac{r_0 \varepsilon}{4 \delta_0} )d(t)^2
\end{align}
for $t \in\left(t_j, t_{j+1}\right)$, where $\delta_0>0$ is to be chosen later and $\delta^*$ is given in \eqref{deltastar}. Recalling \eqref{lyapunov}, \eqref{deltastar}, \eqref{q0} and \eqref{hatlinf1}, there exists $\iota^*(\bar{\lambda})$ such that, for all $\iota \in\left[0, \iota^*\right]$,
\begin{align}
	\dot{V} \leq-\sigma V,
\end{align}
for $t \in\left(t_j, t_{j+1}\right)$ , where $\sigma$ is given in \eqref{sigma1}, with choosing
\begin{align}
	r_0 = &    \frac{16(\kappa_1 q_1+ \kappa_2  +\kappa_3 q_0 )}{\varepsilon} \\
	\delta_0 < & \frac{r_0 \varepsilon \wp-\kappa_2  -2 \kappa_3 -r_0 \frac{\varepsilon}{4} -r_0 \varepsilon \iota \frac{5}{2}}{r_0 \varepsilon},\\
	\lambda_d \geq & \frac{r_0 \varepsilon}{4 \delta_0}. \label{ld3}
\end{align} 
 Since $V(t)$ is continuous (as $\hat{w}(x,t) $ and $m(t)$ are continuous), we have that 
\begin{align}
	V\left(t_{j+1}\right) \leq e^{-\sigma\left(t_{j+1}-t_j\right)} V\left(t_j\right).
\end{align}
Hence, for any $t \geq 0$ in $t \in\left[t_j, t_{j+1}\right), j \in \mathbb{N}$,  by recursion we derive:
\begin{align}
	V(t)  \leq e^{-\sigma \left(t-t_j\right)} V\left(t_j\right) \leq e^{-\sigma t} V(0). \label{Vte}
\end{align}
Using \eqref{ukw} and recalling \eqref{Omega}, we can derive
\begin{align}
	\Omega(t) \leq& \left(1+\|\hat{k}\|_{\infty}\right)^2 \left(1+\|\hat{l}\|_{\infty}\right)^2  e^{-\sigma t} \Omega(0).
\end{align}
 Recalling \eqref{42q} and \eqref{hatlinf1}, the second property of Theorem \ref{theo1} is thus obtained.
\section{Numerical Simulations}\label{sec6}
In the simulation example, we consider the reaction-diffusion PDE with $\varepsilon=1 ;\ \lambda(x)=50\cos(8\cos^{-1}x); \ q=10$, under the initial conditions $u(x,0)=  \cos(\pi x)$. To learn the mapping $\mathcal{K}: \lambda(x) \mapsto k(x, y)$, we use the way to implement a CNN for the DeepONet branch network given in \cite{krstic2024neural}. We show the analytical kernel $k(x,y)$, learned NO-approximated kernel $\hat{k}(x,y)$, and $k(x,y)-\hat{k}(x,y)$ in Figs. \ref{fig01}--\ref{fig03}. \par
The parameters for the event-triggered mechanism are chosen as follows: $m(0)=-5,\ \xi=55,\ \eta=9.775$. To satisfy \eqref{kp1} we choose $\kappa_1=5.5\times10^{4}, \kappa_2=758, \kappa_3=1240$. Then, we choose $\lambda_d =770$ to satisfy \eqref{ld3}.\par
The numerical simulation is conducted by the finite difference method. The plant is discretized with uniform step sizes of $\Delta x=0.02$ for the space variable and $\Delta t=0.0001 $ for the time variable. The open-loop results of the reaction-diffusion system are shown in Fig. \ref{fig2}, from which we observe that the plant is open-loop unstable. The piecewise-constant control input $U_d$ defined in \eqref{c16} and the continuous-in-time control signal $U_f(t)$ used in ETM are shown in Fig. \ref{fig3}. The minimal dwell time in this case is $0.0134$s (the time step in the simulation is $0.0001$s). With the proposed event-triggered controller $U_d$ \eqref{c16}, it is shown in Figs. \ref{fig4}  that  PDE states $u(x, t)$   are convergent to zero.
 \begin{figure}[!t]
	\centering
	\subfloat[]{
		\includegraphics [width=2.5cm] {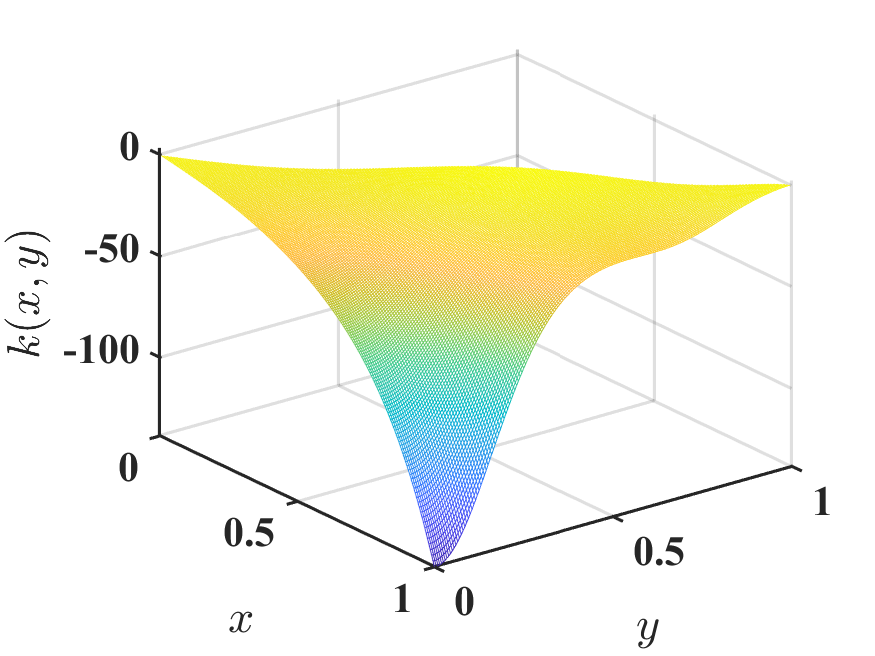}
		\label{fig01}	
	}
	\subfloat[]{
		\includegraphics [width=2.5cm] {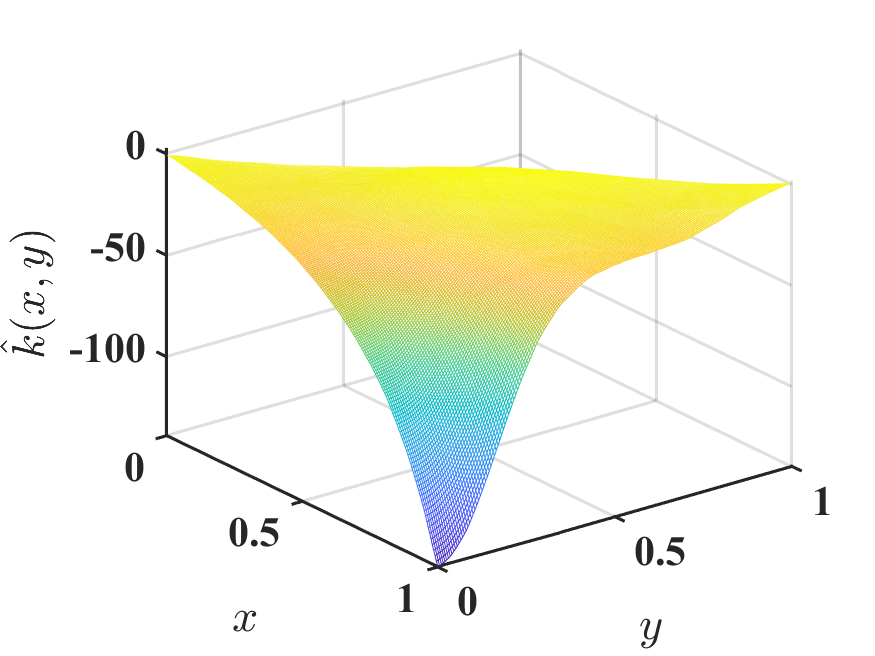}
		\label{fig02}	
	}
    \subfloat[]{
		\includegraphics [width=2.5cm] {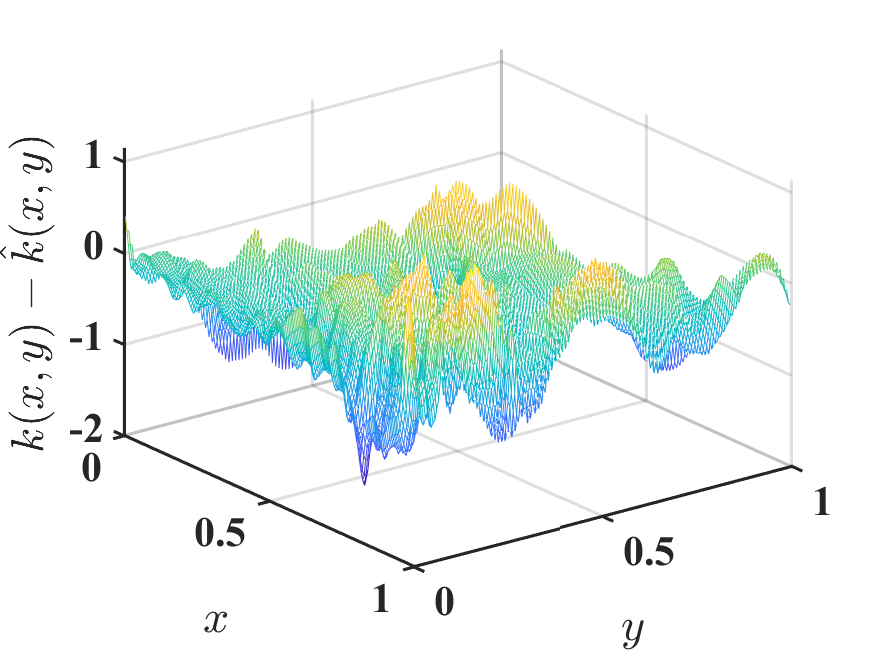}
		\label{fig03}	
	}
	\caption{Results of the kernel $k(x, y)$, learned kernel $\hat{k}(x,y)$, and the kernel error $k(x,y)-\hat{k}(x,y)$.}
\end{figure} 
\begin{figure}
	\centering
	\includegraphics[width=4cm]{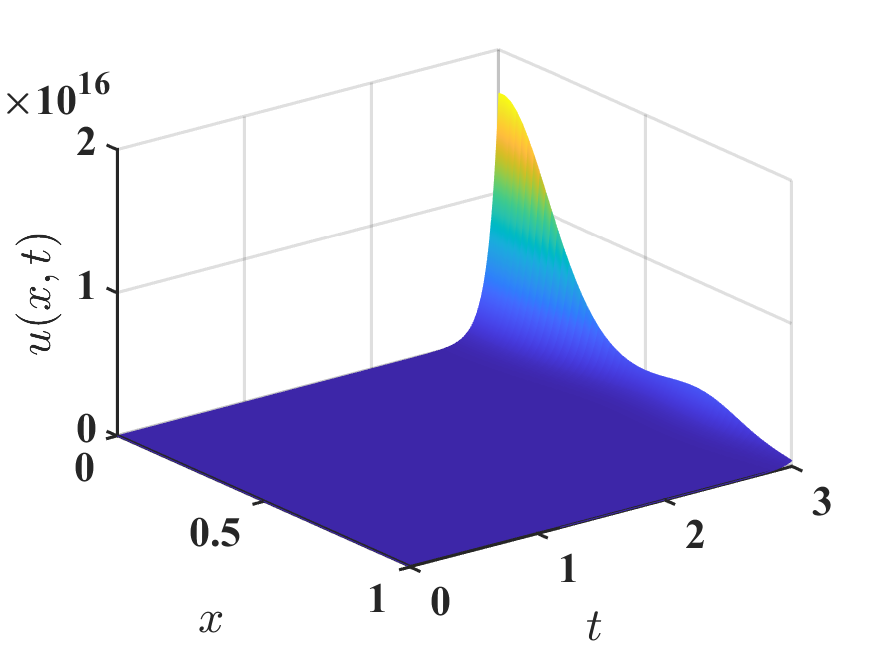}
	\caption{Results for open-loop system.}
	\label{fig2}
\end{figure}
 \begin{figure}[!t]
	\centering
	\subfloat[The ETC signal $U_d$ considered along with the CTC signal used in ETM.]{
		\includegraphics [width=4cm] {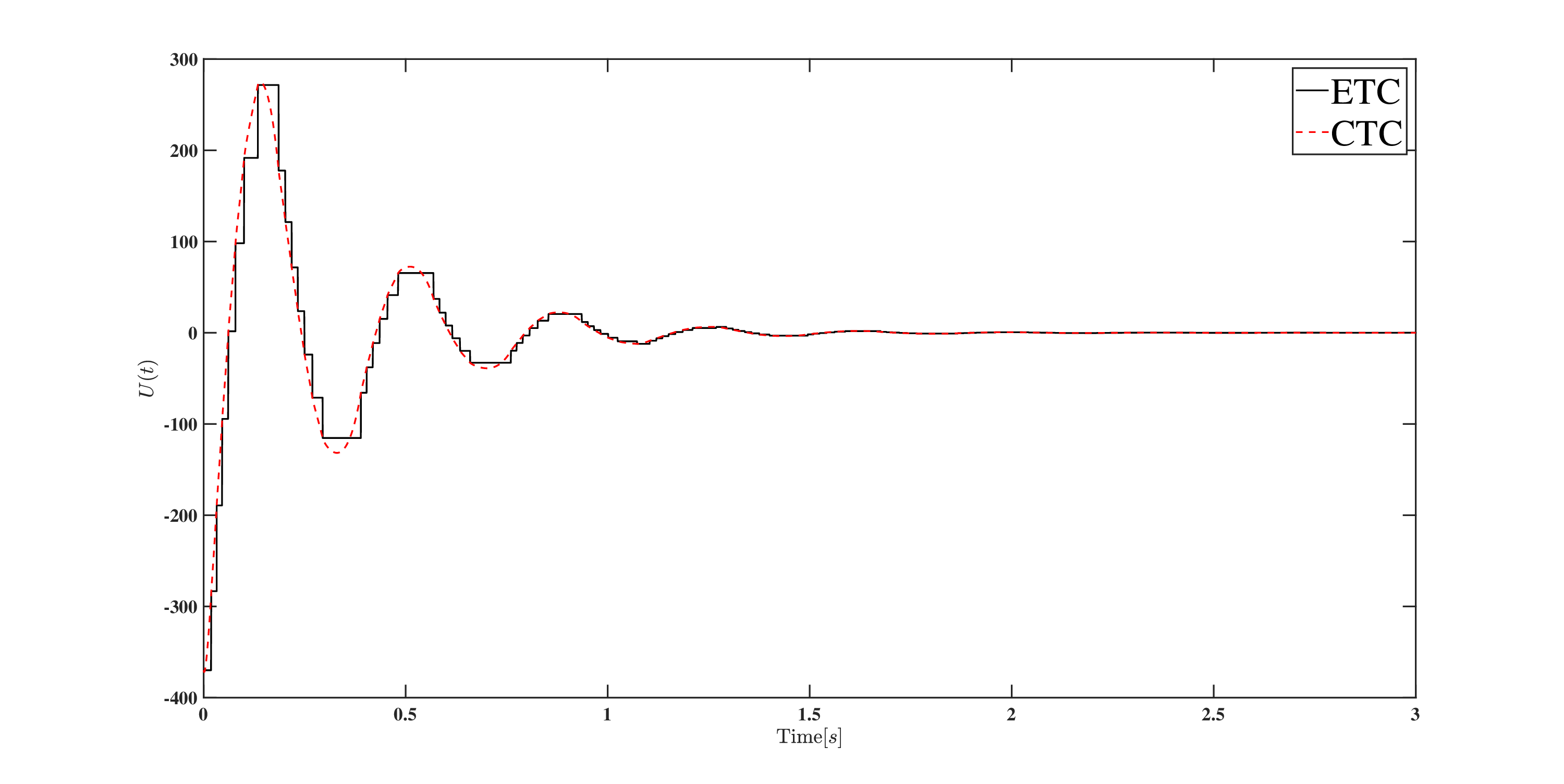}
		\label{fig3}	
	}
	\subfloat[Results for $u(x,t)$ .]{
		\includegraphics [width=4cm] {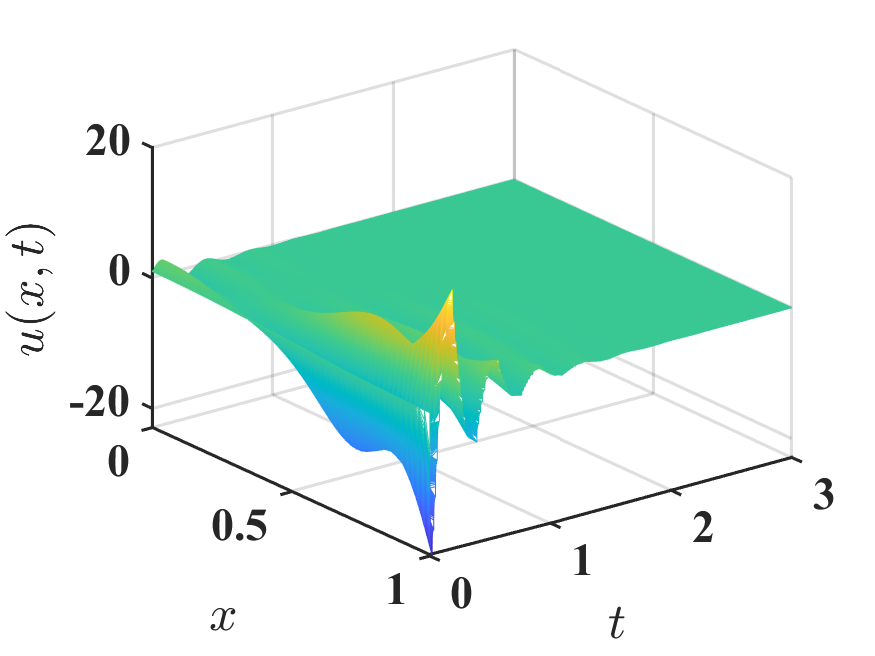}
		\label{fig4}	
	}
	\caption{Results under the event-triggered NO-approximated control input $U_d$.}
\end{figure}

\section{Conclusion}\label{sec7}
In this paper, we propose a NO-approximated event-triggered  
boundary control scheme for a parabolic PDE. We have proved that the proposed control guarantees:
1) no Zeno phenomenon occurs; 2) the plant states are exponentially
convergent to zero. The effectiveness of the proposed design
is verified by a numerical example.

	\bibliography{reference}
\bibliographystyle{plain}

\end{document}